\documentclass[11pt]{amsart}
\usepackage{amsmath,amssymb,latexsym, dsfont}
\usepackage[small]{caption}
\usepackage{graphicx,color,mathrsfs,tikz}
\usepackage{subfigure,color,bm}
\usepackage[colorlinks=true,urlcolor=blue,
citecolor=red,linkcolor=blue,linktocpage,pdfpagelabels,
bookmarksnumbered,bookmarksopen]{hyperref}
\usepackage[italian,english]{babel}
\usepackage{units}
\usepackage{enumitem}
\usepackage[left=2.5cm,right=2.5cm,top=2.9cm,bottom=2.9cm]{geometry}
\usepackage[hyperpageref]{backref}
\usepackage{fancyhdr, dsfont}
\usepackage{tikz}
\usepackage{hyperref,cleveref}
\usepackage{color}

\usepackage{amscd}
\usepackage{esint}

\newcommand{\eps}{\varepsilon}

\theoremstyle{plain}

\newtheorem{thm}{Theorem}[section]
\newtheorem*{thm*}{Theorem}
\theoremstyle{plain}
\newtheorem{lem}[thm]{Lemma}

\newtheorem{cor}[thm]{Corollary}

\theoremstyle{definition}

\newtheorem{rem}{Remark}[section]
\newtheorem{remark}{Remark}[section]
\newtheorem{proposition}{Proposition}[section]

\newcommand{\cD}{{\mathcal D}}

\newcommand{\mR}{\mathbb{R}}

\newtheorem{lemma}{Lemma}[section]

\newcommand{\ro}{\mathbb{R}}

\newcommand{\md}[1]{\left|#1 \right|}
\newcommand{\nrm}[1]{\left|\left| #1 \right|\right|}
\newcommand{\bct}[1]{\left(#1\right)}

\newcommand{\Sb}[1]{\left \lbrace #1\right\rbrace}

\newcommand{\epss}{\epsilon}

\newcommand{\dsp}{\displaystyle}

\newcommand{\R}{\mathbb{R}}

\numberwithin{equation}{section} \allowdisplaybreaks

        \title[Interpolation Inequalities]{Gagliardo-Nirenberg inequalities in fractional Coulomb-Sobolev spaces for radial functions}

        \author[A. Mallick]{Arka Mallick}
        \address[A. Mallick]{Department of Mathematics \newline\indent
	IISc, Bangalore, India}
\email{arkamallick@iisc.ac.in}

        \author[H.-M. Nguyen]{Hoai-Minh Nguyen}
  \address[H.-M. Nguyen]{Laboratoire Jacques Louis Lions \newline\indent
	Sorbonne Universit\'e \newline\indent
	4 Place Jussieu, 75252, Paris, France}
\email{hoai-minh.nguyen@sorbonne-universite.fr}

\begin{document}

\begin{abstract}  
We extend the range of parameters associated with the Gagliardo-Nirenberg interpolation inequalities in the fractional Coulomb-Sobolev spaces for  radial functions. 
We also study the optimality of this newly extended range of parameters.
\end{abstract}

\maketitle
\tableofcontents


\section{Introduction}
 
 In this article we continue our study of Gagliardo-Nirenberg(GN) interpolation inequality in the Coulomb-Sobolev spaces  \cite{NguMalJFA}. Developing the method in \cite{NgSob3}, we proved  the following result in \cite{NguMalJFA}. 
 
 \begin{thm}$($\cite[Theorem 1.1]{NguMalJFA}$) $\label{thm1}
 Let $d \ge 1$,  $\gamma>1$, $1\le p, \, q  < + \infty$, $0 \le s \le 1$, $\alpha \in (0, d)$ and $0 \le  \beta_1, \beta_2 < + \infty $ be such that 
\begin{equation}
\beta_1p+2\beta_2q =1,  \quad (d-sp)\beta_1+(d+\alpha)\beta_2 = d/ \gamma, \label{betaEquations} 
\end{equation}
and
\begin{equation}
\beta_1 \gamma +\beta_2\gamma\geq 1. \label{ConCondtn}
\end{equation}
There exists a constant $C>0$ such that 
\begin{equation}\label{CSF1-*}
\| g \|_{L^\gamma(\ro^d)} \leq C \| g\|_{\dot W^{s, p}(\mR^d)}^{\beta_1 p} \bct{\int_{\ro^d}\int_{\ro^d} \frac{|g(x)|^q|g(y)|^q}{|x-y|^{d-\alpha}}dxdy}^{\beta_2} \mbox{ for all } g \in C^1_c(\mR^d).   
\end{equation} 
 \end{thm}

In \Cref{thm1} and in what follows, for any open set $\Omega \subset \mR^d$, the following notation is used:  
\begin{equation}
\| g \|_{\dot W^{s, p}(\Omega)} =  \left\{\begin{array}{cl} \dsp \bct{\int_{\Omega}\int_{\Omega} \frac{|g(x)-g(y)|^p}{|x-y|^{d+sp}}dxdy}^{\frac{1}{p}} & \mbox{ for } 0 < s < 1, \\[6pt]
\dsp \left( \int_{\Omega} |\nabla g (x)|^p \, dx \right)^{1/p} &  \mbox{ for } s = 1, \\[6pt]
\dsp \left( \int_{\Omega} |g (x)|^p \, dx \right)^{1/p} & \mbox{ for } s =0. 
\end{array}\right.
\end{equation}

Condition \eqref{betaEquations} is due to the scaling invariance of \eqref{CSF1-*} which automatically makes them optimal whereas the optimality of condition \eqref{ConCondtn} was proved in \cite{NguMalJFA}.  Set
\begin{equation} \label{cD}
\cD  = p(d+\alpha)-2q(d-sp). 
\end{equation}
In the case $\cD \neq 0$, one can compute  $\beta_1$ and $\beta_2$ as a function of $d, \, \alpha, \, p, \, q, \, s, \gamma$ by 
 \begin{equation}\label{betaValue}
 \beta_1 = \frac{\gamma(d+\alpha)-2qd}{\gamma\big(p(d+\alpha)-2q(d-sp) \big)} \quad \mbox{ and } \quad \beta_2 = \frac{pd-\gamma(d-sp)}{\gamma\big(p(d+\alpha)-2q(d-sp) \big)}.
 \end{equation}

Gagliardo-Nirenberg inequalities in different function spaces play a crucial role in the study of non-linear PDEs. First instance of such an inequality appeared in works of  Gagliardo and Nirenberg \cite{Gagliardo59, Nirenberg59}. 
In the context of Coulomb-Sobolev spaces, which are relevant in Thomas--Fermi--Dirac--von Weizs\"acker models of density functional theory \cite{BBL81, LL05, Lieb81} or  in Hartree--Fock theory \cite{CLL01, LS10, CDSS13}, the study of GN type inequalities was initiated by Lions \cite{Lions81, pllCMP} in connection to the study of the Hartree--Fock  equation. In fact, Lions proved \eqref{CSF1-*} for $\gamma=3$, $p=q=\alpha=2$, $d=3$ and $s=1$.
Subsequently many extensions of  \eqref{CSF1-*} have been established. One was derived by Bellazzini, Frank, Visciglia \cite{BelFrVis} where they proved \eqref{CSF1-*} in the case $p = 2$, $q =2$, and $0 < s < 1$ (see also \cite[(21)]{LunNamPor} for the case $p=q=2$, $\alpha = d - 2 s$). This was extended by Mercuri, Moroz, and Van Schaftingen \cite{MerMorSch} to the case $p=2$ and $s=1$. Bellazzini, Ghimenti, Mercuri, Moroz, and Van Schaftingen  \cite{MorSch'18} then extended to the case $p=2$ and $0 < s <1$.

\medskip 
A related context is the one of Caffarelli-Kohn-Nirenberg (CKN) inequalities. More precisely,   let $d\geq 1$, $0<s<1$, $p \ge 1$, $q\geq 1$, $\tau \ge 1$, $0 <  a \leq 1$,  and $\alpha_1, \alpha_2,\beta, \gamma \in \ro$. Set $\widetilde{\alpha} = \alpha_1+\alpha_2 $ and define $\sigma$ by $\gamma = a \sigma + (1-a) \beta$.  Assume that 
\begin{align}\label{balance-law2}
\frac{1}{\tau}+\frac{\gamma}{d}= a\bct{\frac{1}{p}+ \frac{\widetilde{\alpha}-s}{d}}+(1-a)\bct{\frac{1}{q}+\frac{\beta}{d}},
\end{align}
and the following conditions hold 
\begin{equation}\label{cond-s-1}
0 \leq \widetilde{\alpha}-\sigma
\end{equation}
and 
\begin{equation}\label{cond-s-2}
\widetilde{\alpha}-\sigma\leq s \text{ if } \frac{1}{\tau}+\frac{\gamma}{d}= \frac{1}{p}+\frac{\widetilde{\alpha}-s}{d}.
\end{equation}
Nguyen and Squassina \cite{HmnSqa} proved that if $\frac{1}{\tau}+\frac{\gamma}{d}>0$,  then there exists some positive constant $C$ such that   
\begin{multline}\label{CKN-fractional1-*}
\| |x|^\gamma g \|_{L^\tau(\ro^d)} \\[6pt]
\leq C \bct{\int_{\ro^d}\int_{\ro^d} \frac{|g(x)-g(y)|^p|x|^{\alpha_1p}|y|^{\alpha_2p}}{|x-y|^{d+sp}}dx dy }^\frac{a}{p}\left \| |x|^\beta g \right\|^{1-a}_{L^q(\ro^d)} \; \; \forall \,  g \in C^1_c(\R^d). 
\end{multline}
This extends the full range of parameters of the well-known CKN inequalities due to Caffarelli, Kohn, and Nirenberg \cite{CKN} (see also \cite{CKN-original}) for $s=1$ to the fractional Sobolev spaces ($0< s < 1$).  Concerning the radial case, 
the range of parameters are larger. Let $C^1_{c, \rm{rad}}(\ro^d)$ denotes the space of all radial, continuously differentiable and compactly supported functions defined in $\mR^d$ with $d \ge 1$. We prove in \cite{HmnMalCRAS} the following result.

\begin{thm} \label{thmMN}$($\cite[Theorem 1.1]{HmnMalCRAS}$)$ Let $d \ge 2$ and assume \eqref{balance-law2}, $\frac{1}{\tau}+\frac{\gamma}{d}>0$, and  
\begin{equation}\label{cond-s-1-rad}
- (d-1) s  \leq \widetilde{\alpha}-\sigma < 0. 
\end{equation}
Then
\begin{multline}\label{CKN-fractional1}
\| |x|^\gamma g \|_{L^\tau(\ro^d)} \\[6pt]
\leq C \bct{\int_{\ro^d}\int_{\ro^d} \frac{|g(x)-g(y)|^p|x|^{\alpha_1p}|y|^{\alpha_2p}}{|x-y|^{d+sp}}dx dy }^\frac{a}{p}\left \| |x|^\beta g \right\|^{1-a}_{L^q(\ro^d)}, \; \; \forall \,  g \in C^1_{c, \rm{rad}} (\R^d). 
\end{multline}
\end{thm}

\begin{remark}
Related results of \eqref{CKN-fractional1-*} and \eqref{CKN-fractional1} in the case $\frac{1}{\tau}+\frac{\gamma}{d} \le 0$ are also studied in \cite{HmnSqa,HmnMalCRAS}. 
\end{remark}

In this article, we are in pursuit of an optimal version of the inequality \eqref{CSF1-*} for radial functions in the spirit of \cite{HmnMalCRAS}.  Recall that $\cD$ is defined in \eqref{cD}. Here is the first result of the paper. 

\begin{thm}\label{thm1-rad}
Let $d \ge 2$, $0< s \le 1$, $1<\gamma<+\infty, \ 1\le p, \, q < + \infty$, $1 <\alpha<d$, $0< \beta_1, \, \beta_2 < + \infty$ be such that \eqref{betaEquations}  holds and $\cD \neq 0$. Assume that either 
\begin{equation}\label{thm1-rad-cd1}
\beta_1 \gamma + \frac{d+\alpha-2}{d-1} \beta_2 \gamma    >  1 
\end{equation}
or  
\begin{equation}\label{thm1-rad-cd2}
\left( \beta_1 \gamma + \frac{d+\alpha-2}{d-1} \beta_2 \gamma = 1  \mbox{ and }   q(1-sp)=p \right).
\end{equation}
Then 
\begin{equation}\label{CSF1}
\| g\|_{L^\gamma(\ro^d)} \leq C \| g\|_{\dot W^{s, p}(\mR^d)}^{\beta_1 p} \bct{\int_{\ro^d}\int_{\ro^d} \frac{|g(x)|^q|g(y)|^q}{|x-y|^{d-\alpha}}dxdy}^{\beta_2}, \mbox{ for all } g \in C^1_{c, \rm{rad}}(\R^d). 
\end{equation} 
\end{thm}

\begin{remark} \Cref{thm1-rad} deals with the case $1 < \alpha < d$ and extends \Cref{thm1} in this case for radial functions under the assumption that $\cD \neq 0$.  One cannot extend \Cref{thm1} for radial functions in the case $0 < \alpha \le 1$ (see \Cref{thm1-rad-opt1} below).  
\end{remark}

\begin{remark} Condition \eqref{thm1-rad-cd2} implicitly implies that $sp < 1$. 
\end{remark}

Our next result addresses the optimality of the range \eqref{thm1-rad-cd1}-\eqref{thm1-rad-cd2}. 

\begin{thm}\label{thm1-rad-opt}
Let $d \ge 2$, $0< s \le 1$, $1< \gamma <+\infty , \,1\le p, \, q < + \infty$, $1 <\alpha<d$, $0 <  \beta_1, \, \beta_2 < + \infty$ be such that \eqref{betaEquations} holds and $\cD \neq 0$. Assume that either 
\begin{equation}\label{thm1-rad-cd1-Opt}
 \beta_1 \gamma + \frac{d+\alpha-2}{d-1} \beta_2 \gamma    < 1 
\end{equation}
or
\begin{equation}\label{thm1-rad-cd2-Opt}
\left( \beta_1 \gamma + \frac{d+\alpha-2}{d-1} \beta_2 \gamma = 1 \mbox{ and }  q(1-sp)\neq p \right).
\end{equation}
Then \eqref{CSF1} does {\bf not} hold. 
\end{thm}

We also obtain the following result which is on the optimality of \eqref{ConCondtn} when $0 < \alpha \le 1$. 

\begin{thm}\label{thm1-rad-opt1}
Let $d \ge 2$, $0< s \le 1$, $1<\gamma <+\infty, \,1\le p, \, q < + \infty$, $0<\alpha\le1$, $0 < \beta_1, \, \beta_2 < + \infty$ be such that \eqref{betaEquations} holds and $\cD \neq 0$. Assume that
\begin{equation}\label{thm1-rad-cd1-Opt1}
 \beta_1 \gamma +  \beta_2 \gamma < 1. 
 \end{equation}
Then \eqref{CSF1} {\bf fails} to hold. 
\end{thm}

\begin{remark} \rm \Cref{thm1-rad-opt1} confirms the optimality of \eqref{ConCondtn} even for radial functions. 
\end{remark}

Various special cases of \Cref{thm1-rad} are known in the literature where the assumptions were written in a quite involved manner and only for the case $p=2$. More precisely, when $p=2$, the conclusion of \Cref{thm1-rad} was proved under the following assumption on $\gamma$ instead of the assumption \eqref{thm1-rad-cd1}-\eqref{thm1-rad-cd2}
\begin{align}\label{gammaRadRangeProof1}
\begin{cases}
\dsp q+ \frac{\bct{q(2s-1)+2}(d-\alpha)}{2s(d+\alpha-2)+(d-\alpha)}< \gamma <\infty, &\text{ if } s \geq \frac{d}{2}, \\[6pt]
\dsp \gamma\in \Bigg(q+ \frac{\bct{q(2s-1)+2}(d-\alpha)}{2s(d+\alpha-2)+(d-\alpha)}, \frac{2d}{d-2s}\Bigg], &\text{ if } s<\frac{d}{2} \text{ and } \frac{1}{q}> \frac{d-2s}{d+\alpha}, \\[6pt] 
\dsp \gamma\in \Bigg[\frac{2d}{d-2s},q+ \frac{\bct{q(2s-1)+2}(d-\alpha)}{2s(d+\alpha-2)+(d-\alpha)} \Bigg), &\text{ if } s<\frac{d}{2} \text{ and } \frac{1}{q}< \frac{d-2s}{d+\alpha} \text{ and } \frac{1}{q}\neq \frac{1-2s}{2},\\[6pt]
\dsp \frac{2d}{d-2s} \leq \gamma \leq q, &\text{ if } s<\frac{1}{2} \text{ and } \frac{1}{q}=\frac{1-2s}{2}. 
\end{cases}
\end{align}
It can be shown, with the help of  \Cref{pro1-thm1-rad-proof} in \Cref{imRanRad}, that  \eqref{thm1-rad-cd1}-\eqref{thm1-rad-cd2} is equivalent to \eqref{gammaRadRangeProof1} when $p=2$. Our results are new even in the case $s=1$ and $p \neq 2$, to our knowledge, and as it is evident from \eqref{gammaRadRangeProof1}, the assumptions given in this paper have a simple form than known ones. Concerning known results, Ruiz \cite{Rui} established the result for the case $s=1$, $d=3$, $\alpha = 2$,  and $p=q=2$. When $d \ge 2$, Mercuri, Moroz, and Van~Schaftingen \cite{MerMorSch} obtained the result in the case $1 < \alpha < d$,  $q \geq 1$, $p=2$, and $s=1$. The result  was later established to the case $1/2 < s < 1$, for $d=3$ with the same ranges of $\alpha$, $q$, and $p$ by Bellazzini, Ghimenti and Ozawa \cite{BelGhiOz}. Finally, Bellazzini, Ghimenti, Mercuri,  Moroz, Van Schaftingen obtained \Cref{thm1-rad} in the case $p=2$. The optimality discussed in \Cref{thm1-rad-opt} and \Cref{thm1-rad-opt1}  is known in the case $p=2$, a result due to Bellazzini, Ghimenti, Mercuri, Moroz, and Van~Schaftingen 
\cite{MorSch'18}.

\medskip 
We next briefly describe the ideas of the proof of \Cref{thm1-rad}. The idea is to derive useful consequences of the radial improvements of the CKN inequalities \Cref{thmMN} and a point wise estimate of radial functions related to Strauss' lemma.  These results are given in \Cref{CKN-for-Rad}.  We then apply these results to some suitable approximation $\gamma_\eps$ of $\gamma$ and then interpolation inequalities are involved. At some point, we also use an estimate of the Coulomb energy due to Ruiz \cite{Rui} (see \Cref{ruiz}).

\medskip 
The paper is organized as follows. \Cref{imRanRad} is devoted to the proof of \Cref{thm1-rad} and consists of two subsections. In the first one, we derive various versions of CKN inequalities for radial functions which will be useful in the proof of \Cref{thm1-rad}. The proof of \Cref{thm1-rad} is given in the second one.  In \Cref{sect-Opt}, we discuss the optimality of the parameters. \Cref{thm1-rad-opt} and \Cref{thm1-rad-opt1} are established there.

 \section{The Coulomb-Sobolev inequality for radial functions}\label{imRanRad}

This section consists of two subsections and is devoted to prove Theorem \ref{thm1-rad}.

\subsection{Preliminaries} \label{CKN-for-Rad}
In this section, we establish various results used in the proof of \Cref{thm1-rad}. Recall that $\cD$ is defined in \eqref{cD}. 
The following first result is a consequence of \Cref{thmMN} and  
is the main ingredient for the proof of \Cref{thm1-rad}.

\begin{lemma}\label{Rad}
Let $d \ge 2$, $0<s\leq1$, $1\leq p,q<\infty$, and $0<\alpha<d$. There exists a constant $C>0$ such that the following two assertions hold. 

\begin{itemize}
\item[(i)] If $\cD >0$, then for all $\eps>0$ small enough such that $\cD - 2\eps p = p(d+\alpha-2\epsilon)-2q (d-sp)>0$, it holds
\begin{align}\label{plsEpss}
\nrm{g}_{L^{\gamma_\eps}(\ro^d)} \leq C \nrm{g}_{\dot{W}^{s,p}(\ro^d)}^{a_\eps} \nrm{|x|^{-(\frac{d-\alpha}{2q}+\frac{\eps}{q})}g}^{1-a_\eps}_{L^q(\ro^d)}, \ \forall g\in C_{c,\rm{rad}}^1(\ro^d),
\end{align}
where $\gamma_\eps$ and $a_\eps$ are defined by 
\begin{align}\label{gammaRadEps1}
\gamma_\eps &= q+\frac{\big(q(sp-1)+p \big)(d-\alpha+2\eps)}{sp(d+\alpha-2\epsilon-2) +(d-\alpha+2\epss)} \mbox{ and }a_\eps = \frac{p\big(\gamma_\eps(d+\alpha-2\epsilon) -2qd \big)}{\gamma_\eps\big(p(d+\alpha-2\epss)-2q(d-sp) \big)}, \end{align}

\item[(ii)] and if $\cD <0$ then for all $\eps>0$ small enough such that $\cD + 2 \eps p = p(d+\alpha+2\epsilon)-2q (d-sp)<0$, it holds
\begin{align}\label{mnsEpss}
\nrm{g}_{L^{\gamma_\eps}(\ro^d)} \leq C \nrm{g}_{\dot{W}^{s,p}(\ro^d)}^{a_\eps} \nrm{|x|^{-(\frac{d-\alpha}{2q}-\frac{\eps}{q})}g}^{1-a_\eps}_{L^q(\ro^d)}, \ \forall g\in C_{c,\rm{rad}}^1(\ro^d),
\end{align}
where $\gamma_\eps$ and $a_\eps$ are defined by 
\begin{align}\label{gammaRadEps2}
\gamma_\eps &= q+\frac{\big(q(sp-1)+p \big)(d-\alpha-2\eps)}{sp(d+\alpha+2\epsilon-2) +(d-\alpha-2\epss)} \mbox{ and }  \ a_\eps = \frac{p \big(\gamma_\eps(d+\alpha+2\epsilon) -2qd \big)}{\gamma_\eps\big(p(d+\alpha+2\epss)-2q(d-sp) \big)}.
\end{align}

\end{itemize}

\end{lemma}

\begin{remark} \rm The signs in front of $2 \eps$ of the corresponding terms in \eqref{gammaRadEps1} and \eqref{gammaRadEps2} are opposite. 
\end{remark}

\begin{proof}
We only prove the assertion $(i)$. Assertion $(ii)$ follows similarly. 

We have, by \eqref{gammaRadEps1}, 
\begin{align*}
   \gamma_\eps - \frac{2qd}{d+\alpha-2\eps} &= q+\frac{\big(q(sp-1)+p \big)(d-\alpha+2\eps)}{sp(d+\alpha-2\epsilon-2) +(d-\alpha+2\epss)} 
   - \frac{2qd}{d+\alpha-2\eps} \\
 &= \frac{\big(qsp- q +p \big)(d-\alpha+2\eps)}{sp(d+\alpha-2\epsilon-2) +(d-\alpha+2\epss)} 
   - \frac{q(d - \alpha + 2 \eps)}{d+\alpha-2\eps}. 
   \end{align*}
This implies 
\begin{equation}\label{bta1Pos}
   \gamma_\eps - \frac{2qd}{d+\alpha-2\eps} 
   =\frac{\big(p(d+\alpha-2\eps)-2q(d-sp) \big)(d-\alpha+2\eps)}{(d+\alpha-2\eps)\big(sp(d+\alpha-2\eps-2)+(d-\alpha+2\eps)\big)}>0, 
\end{equation}
by the smallness of $\eps$. 

Using the definition of $a_\eps$ in \eqref{gammaRadEps1}, we derive  from \eqref{bta1Pos} and the smallness of $\eps$ that 
\begin{equation}
a_\eps>0. 
\end{equation}
We have, by \eqref{gammaRadEps1},  
\begin{equation} \label{1-aeps}
1-a_\eps = 1 - \frac{p\big(\gamma_\eps(d+\alpha+2\epsilon) -2qd \big)}{\gamma_\eps\big(p(d+\alpha+2\epss)-2q(d-sp) \big)} = \frac{2q \big(dp-\gamma_\eps(d-sp) \big)}{\gamma_\eps \big(p(d+\alpha-2\eps)-2q(d-sp) \big)},
\end{equation}
which yields 
$$
a_\eps < 1 \mbox{ if } sp\geq d. 
$$
We next deal with the case $sp < d$. From \eqref{gammaRadEps1}, we derive that 
\begin{equation}
\gamma_{\eps}= \frac{2qsp(d-1)+p(d-\alpha+2\eps)}{sp(d+\alpha-2\eps-2)+ (d-\alpha+2\eps)}. \label{AC2}
\end{equation}
Using \eqref{AC2}, we obtain  
\begin{align*}
\frac{dp}{d-sp}- \gamma_\eps &= \frac{dsp^2(d+\alpha-2\eps-2)+dp(d-\alpha+2\eps)}{(d-sp)\big(sp(d+\alpha-2\eps-2)+ (d-\alpha+2\eps)\big)} \\
& \quad \quad \quad \quad  - \frac{2qsp(d-1)(d-sp) + p(d-sp)(d-\alpha+2\eps)}{(d-sp)\big(sp(d+\alpha-2\eps-2)+ (d-\alpha+2\eps) \big)}
\\ &= \frac{dsp^2(d+\alpha-2\eps-2)+sp^2(d-\alpha+2\eps)-2qsp(d-1)(d-sp)}{(d-sp)\big(sp(d+\alpha-2\eps-2)+ (d-\alpha+2\eps) \big)}\\&= \frac{sp^2(d-1)(d+\alpha-2\eps-2)+2sp^2(d-1)-2qsp(d-1)(d-sp)}{(d-sp)\big(sp(d+\alpha-2\eps-2)+ (d-\alpha+2\eps) \big)}\\&= \frac{sp(d-1)\bct{p(d+\alpha-2\eps-2)+2p-2q(d-sp)}}{(d-sp)(sp(d+\alpha-2\eps-2)+ (d-\alpha+2\eps))}, 
\end{align*}
which yields, for $sp < d$,
\begin{align}\label{bta2Pos}
\frac{dp}{d-sp} -\gamma_\eps = \frac{sp(d-1)\big(p(d+\alpha-2\eps) -2q(d-sp) \big)}{(d-sp) \big(sp(d+\alpha-2\eps-2)+(d-\alpha+2\eps)\big)}>0, 
\end{align}
by the smallness of $\eps$,    which implies, by \eqref{gammaRadEps1}, that 
$$
a_\eps < 1. 
$$   
We thus established  
$$
0<a_\eps<1.
$$

In light of \eqref{CKN-fractional1} under the assumptions \eqref{balance-law2} and \eqref{cond-s-1-rad}, it suffices to verify these conditions for the following parameters $(s, p, q, \tau, \gamma, a, \alpha_1, \alpha_2, \beta) = (s, p, q, \gamma_\eps, 0, a_\eps, 0, 0, - \frac{d-\alpha + 2 \eps}{2q})$. 
One can check that 
\begin{equation} \label{Rad-alpha}
\widetilde{\alpha}=\alpha_1+\alpha_2=0
\end{equation}
and 
\begin{equation} \label{Rad-sigma}
\sigma = \frac{1}{a} \big(\gamma  - (1-a) \beta \big) = 
\frac{1-a_\eps}{a_\eps}\frac{d-\alpha+2\eps}{2q}.
\end{equation}
From \eqref{gammaRadEps1} and \eqref{1-aeps}, we deduce that
\begin{align*}
&\frac{d-sp}{d} \frac{\gamma_\eps a_\eps}{p}+ \frac{d+\alpha-2\eps}{d} \frac{\gamma_\eps(1-a_\eps)}{2q}\\ &= \frac{(d-sp)\big(\gamma_\eps(d+\alpha-2\epsilon) -2qd \big) + (d+\alpha-2\eps) \big(dp-\gamma_\eps(d-sp) \big)}{d\big(p(d+\alpha-2\eps)-2q(d-sp)\big)}=1. 
\end{align*}
This implies 
\begin{equation*}
\frac{1}{\gamma_\eps}=  \frac{a_\eps (d-sp) }{d p} + \frac{(1-a_\eps)(d+\alpha-2\eps)}{2 d q} \notag
\end{equation*}
and thus \eqref{balance-law2} holds. 

We have
\begin{equation}\label{alphaMinusSigma}
\widetilde{\alpha}-\sigma \mathop{=}^{\eqref{Rad-alpha}, \eqref{Rad-sigma}} - \frac{1-a_\eps}{a_\eps} \frac{d-\alpha+2\eps}{2q} \mathop{=}^{\eqref{gammaRadEps1}, \eqref{1-aeps}} -\frac{\big(dp-\gamma_\eps(d-sp) \big)(d-\alpha+2\eps)}{p\big(\gamma_\eps(d+\alpha-2\epsilon) -2qd \big)}.
\end{equation}
Plugging \eqref{bta1Pos} and \eqref{bta2Pos} in \eqref{alphaMinusSigma}, we get 
$$
\widetilde{\alpha}-\sigma= -(d-1)s.
$$ 
Therefore, \eqref{cond-s-1-rad} is satisfied. 

\medskip 
Thus, the assumptions \eqref{balance-law2} and \eqref{cond-s-1-rad} are checked. The proof is complete.
\end{proof}

Here is another consequence of the CKN inequalities \eqref{CKN-fractional1} for radial functions, which is used in the proof of \Cref{thm1-rad}. 

\begin{lemma}\label{stWss}
Let $d\geq 2$, $1\leq p<\infty$ , $0<s\leq1$ be such that $sp< d$, and $\beta \in \R$.  Assume that 
\begin{itemize}
\item [(a)] $\frac{1}{r} = \frac{1}{p} + \frac{\beta-s}{d}, $
\item[(b)] $- (d-1) s\leq \beta \leq s.$ 
\end{itemize}
There exists a positive constant $C$ such that 
\begin{align}\label{stWssIneq}
\left\| |x|^{-\beta}g \right\|_{L^r(\ro^d)}\leq C \| g \|_{\dot{W}^{s,p}(\ro^d)}, \mbox{ for all } g \in C^1_{c, \rm{rad}}(\R^d).
\end{align}

\end{lemma}

\begin{proof}
First of all notice that, since $\beta \leq s$ we must have $r\geq p$.
In light of \eqref{CKN-fractional1} under the assumptions \eqref{balance-law2}, \eqref{cond-s-1-rad}, and \eqref{cond-s-2}, it suffices  to verify these conditions for  the following parameters $(a, s, p, \tau, \gamma, \alpha_1, \alpha_2) = (1, s, p, r, - \beta, 0, 0)$.  
We compute $\widetilde{\alpha}=\alpha_1+\alpha_2=0$ and $\sigma= -\beta$. Note that by the assumption $(a)$ we have $\frac{1}{\tau}+\frac{\gamma}{d}=\frac{d-sp}{dp} >0$. So, \eqref{balance-law2} is satisfied. On the other hand, \eqref{cond-s-1-rad} and \eqref{cond-s-2}  are consequences of the assumptions $(b)$. We thus have the desired inequality.
\end{proof}

\begin{remark} \rm Particular versions of \eqref{stWssIneq} can be found in \cite{Rubin} (see also of \cite[Theorem 4.3]{MorSch'18} and the discussion afterwards for more references).
\end{remark}

As a consequence of \eqref{stWssIneq}, we obtain the following useful estimate.

\begin{lemma}\label{wkNi}
Let $d\geq 2$, $1\leq p<\infty$, and $0<s\leq1$ be such that $sp\leq 1$. Assume that   $\frac{1}{p}-s \leq \frac{1}{\gamma} \leq \frac{1}{p}-\frac{s}{d}$. Then there exists a positive constant $C$ such that for any $R>0$, we have 
\begin{align}\label{wkNiIneq}
\bct{\int_{|x|>R} |g(x)|^\gamma dx}^\frac{1}{\gamma} \leq CR^{\frac{d}{\gamma}-(\frac{d}{p}-s)} \| g \|_{\dot{W}^{s,p}(\ro^d)}, \mbox{ for all } g\in C_{c,\rm{rad}}^1(\ro^d).
\end{align}
\end{lemma}

\begin{proof} Inequality \eqref{wkNiIneq} is a consequence of \eqref{stWssIneq} with $r = \gamma$ and $\beta = \frac{d}{\gamma} - \left( \frac{d}{p} - s \right)$ after noting that $\beta \le 0$ by the assumptions. 
\end{proof}

We next present a point wise estimate of radial functions related to Strauss' lemma. 

\begin{lemma}\label{StLem}
Let $0<s\le1$, $1\leq p<\infty$ and $d'> 0$ be such that $1<sp<d'$ and assume that $\Lambda>1$. Then there exist a constant $C=C(d',s,p,\Lambda)>0$ such that if $0<s<1$,  then for any $g\in W^{s,p}(\ro)$ it holds
\begin{align}\label{StLem1}
|g(x)| \leq \frac{C}{|x|^\frac{d'-sp}{p}} \bct{\int_{\ro} \int_{\ro} \frac{|g(r)-g(\rho)|^p|r|^{d'-1}}{|r-\rho|^{1+sp}} \chi_{\Lambda}(|r|,|\rho|)dr d\rho}^\frac{1}{p}, \forall x\neq 0,
\end{align}
where, for $r_1, r_2  \ge 0$,  we denote
\begin{equation}\label{def-chi}
\chi_{\Lambda} (r_1, r_2) = \left\{  \begin{array}{cl} 1  &  \mbox{ for  } \Lambda^{-1} r_1 \le r_2 \le \Lambda r_1, \\[6pt]
0 &  \mbox{ otherwise},
\end{array}\right. 
\end{equation}
and if $s=1$, then for any $g\in W^{1,p}(\ro)$ it holds
 \begin{align}\label{StLem1s=1}
 |g(x)| \leq \frac{C}{|x|^\frac{d'-p}{p}} \bct{\int_{\ro} \md{g'(r)}^p |r|^{d'-1} dr  }^\frac{1}{p}, \ \forall x\neq 0.
 \end{align}
\end{lemma}

\begin{proof}
First we consider $0<s<1$. Since $1<sp$, $g$ is H\"older continuous by \cite[Theorem 8.1]{NPV} . Without loss of generality, we can assume that $g$ has compact support and $\mathrm{supp}\ g \subset (-2^m,2^m)$ for some $m\in \mathbb{N}$. Since, \eqref{StLem1} is invariant under scaler multiplication, so  we could assume that 
\begin{align}
 \bct{\int_{\ro} \int_{\ro} \frac{|g(r)-g(\rho)|^p|r|^{d'-1}}{|r-\rho|^{1+sp}} \chi_{\Lambda}(|r|,|\rho|)dr d\rho}^\frac{1}{p}=1  \label{StLem1A}. 
\end{align}

Using a scaling argument, it suffices to prove \eqref{StLem1} for $x=1$. 

By the embedding for the fractional Sobolev spaces, see e.g., \cite[estimates 8.4 and 8.8]{NPV},  we have 
\begin{align}\label{StLem2}
|g(1)-g(2)| &\leq C \bct{\int_{ \frac{1}{2} \le |r| \le 4} \int_{\frac{1}{2} \le |\rho| \le 4} \frac{|g(r)-g(\rho)|^p }{|r-\rho|^{1+sp}} \chi_{\Lambda}(|r|,|\rho|) dr d\rho}^\frac{1}{p} \notag \\ 
& \leq C\bct{\int_{\frac{1}{2}<|r|< 4}\int_{\frac{1}{2}<|\rho|< 4}\frac{|g(r)-g(\rho)|^p|r|^{d'-1}}{|r-\rho|^{1+sp}} \chi_{\Lambda}(|r|,|\rho|) dr d\rho}^\frac{1}{p}. 
\end{align}
Here and in what follows, $C>0$ is a constant depends on $s, d', p$ and $\Lambda$. Let $\lambda>0$ and denote $g_\lambda(x):= g(\lambda x)$. Applying \eqref{StLem2} to $g_\lambda$, we obtain 
\begin{align}\label{StLem3}
\md{g(\lambda)-g(2\lambda)} \leq \frac{C}{\lambda^\frac{d'-sp}{p}} \bct{\int_{\frac{\lambda}{2}<|r|< 4 \lambda}\int_{\frac{\lambda}{2}<|\rho|< 4 \lambda}\frac{|g(r)-g(\rho)|^p|r|^{d'-1}}{|r-\rho|^{1+sp}} \chi_{\Lambda}(|r|,|\rho|) dr d\rho}^\frac{1}{p}.
\end{align}
Using the fact \begin{align*}
|g(1)| = |g(1)-g(2^m)| \leq \sum_{k=1}^m \md{g(2^{k-1}) - g(2^k)}, 
\end{align*}
we derive from \eqref{StLem3} that
\begin{multline}\label{StLem4}
|g(1)| \leq C \sum_{k=1}^m \bct{\frac{1}{2^\frac{k-1}{p}}}^{d'-sp} \\[6pt]
\times  \bct{\int_{2^{k-2} \le |r| \le 2^{k+1}}\int_{2^{k-2} \le |\rho| \le 2^{k+1}}\frac{|g(r)-g(\rho)|^p|r|^{d'-1}}{|r-\rho|^{1+sp}} \chi_{\Lambda}(|r|,|\rho|) dr d\rho}^\frac{1}{p}.
\end{multline}
Since $sp<d'$, so $\sum_{k\geq 1} 2^{-\frac{(k-1)(d'-sp)}{p}}<\infty$,  using  \eqref{StLem1A}  in \eqref{StLem4} we establish \eqref{StLem1}. This completes the proof in the case $0<s<1$. The proof in the case $s=1$ follows similarly. The details are omitted. 
\end{proof}

As a consequence of \Cref{StLem}, we obtain the following result. 

\begin{cor}\label{StLemFrc}
Let $0<s\leq1$, $1\leq p<\infty$ and $d\geq 2$ be such that $1<sp<d$. Then there exists a constant $C>0$ depending on $s,p$ and $d$, such that for any  $g\in C_{c,\rm{rad}}^1(\ro^d)$ we have 

\begin{align}\label{StLemIneq}
|g(x)| \leq \frac{C}{|x|^\frac{d-sp}{p}}\| g \|_{\dot{W}^{s,p}(\ro^d)} ,\ \ \forall x\in \ro^d\setminus \Sb{0}.
\end{align}
\end{cor}
\begin{proof}
\Cref{StLemFrc} is a direct consequence of \Cref{StLem} and  \cite[Lemma 3.1]{HmnMalCRAS}. 
\end{proof}

\begin{remark} \rm
 The inequality \eqref{StLemIneq} is an extension of the radial lemma of Strauss \cite{Strss} to the fractional Sobolev spaces. This type of inequalities has been derived previously in \cite[Theorem 3.1]{Nap} and \cite{ChOz} for $p=2$ and the known proofs are limited to the case $p=2$.
\end{remark}

\subsection{Proof of \Cref{thm1-rad}}
We start by proving the following proposition which allow us to divide the proof of \Cref{thm1-rad} into three different sub cases. 
 
\begin{proposition}\label{pro1-thm1-rad-proof}
Let $d \ge 2$, $0< s \le 1$, $1 \leq \gamma, \, p, \, q < + \infty$, $1 <\alpha<d$, and $0 \le \beta_1, \, \beta_2 < + \infty$, and assume \eqref{betaEquations} and $(d+\alpha)p-2q(d-sp)\neq 0$. 
Set 
\begin{equation}\label{gamaRad}
\gamma_{\mathrm{rad}} := q+ \frac{\bct{q(sp-1)+p}(d-\alpha)}{sp(d+\alpha-2)+(d-\alpha)} =  \frac{2qsp(d-1)+p(d-\alpha)}{sp(d+\alpha-2)+(d-\alpha)}.
\end{equation}
Then \eqref{thm1-rad-cd1} is equivalent to the following:  
\begin{align}\label{gammaRadRangeProof}
\begin{cases}
\dsp \gamma_{\mathrm{rad}} < \gamma \leq \frac{pd}{d-sp}, &\text{ if } \cD > 0 \mbox{ and } sp < d, \\[6pt] 
\dsp \gamma_{\mathrm{rad}}< \gamma <\infty, &\text{ if }  sp \geq d, \\[6pt]
\dsp \frac{pd}{d-sp} \leq \gamma < \gamma_{\mathrm{rad}}, &\text{ if } \cD< 0,
\end{cases}
\end{align}
and \eqref{thm1-rad-cd2} is equivalent to the fact 
\begin{equation}\label{gammaRadRangeProof-2}
\gamma = \gamma_{\rm{rad}} \quad \mbox{ and } \quad q = \frac{p}{1-sp}. 
\end{equation}
\end{proposition}

\begin{proof} We have, by \eqref{betaValue},  
\begin{align*}
\beta_1\gamma +   \frac{\beta_2 \gamma (d+\alpha-2)}{d-1}  - 1 
= &  \frac{\big(\gamma(d+\alpha)-2qd \big) \gamma}{\gamma\big(p(d+\alpha)-2q(d-sp)\big)} + \frac{\big(pd-\gamma(d-sp) \big) \gamma (d+\alpha-2) }{(d-1) \gamma\big(p(d+\alpha)-2q(d-sp)\big)} - 1 \\[6pt]
= &  \frac{(d-1) \big(\gamma(d+\alpha)-2qd \big)  + \big(pd-\gamma(d-sp) \big) (d+\alpha-2)}{(d-1)\big(p(d+\alpha)-2q(d-sp)\big)}  \\[6pt]
& \qquad \qquad  - \frac{(d-1)\big(p(d+\alpha)-2q(d-sp)\big) }{(d-1)\big(p(d+\alpha)-2q(d-sp)\big)}.  
\end{align*}
Simplifying the expression after noting $(d-1)(d+\alpha) - d (d + \alpha - 2) = d- \alpha$, we obtain 
\begin{align}\label{gamma=1Rem}
\beta_1\gamma + \frac{\beta_2 \gamma (d+\alpha-2)}{d-1} - 1
=   \frac{\gamma\big(sp(d+\alpha-2)+(d-\alpha) \big) - \big(2qsp(d-1)+p(d-\alpha) \big)}{(d-1)\big(p(d+\alpha)-2q(d-sp) \big)}, 
\end{align}
which yields, by \eqref{gamaRad}, 
\begin{equation}\label{pro1-thm1-rad-proof-p0}
\beta_1\gamma + \frac{\beta_2 \gamma (d+\alpha-2)}{d-1} - 1 = \frac{sp(d+\alpha-2)+(d-\alpha)}{(d-1)} \frac{\bct{\gamma- \gamma_{rad}}}{\big(p(d+\alpha)-2q(d-sp) \big)}.
\end{equation}

On the other hand, since $\beta_1, \beta_2 \ge 0$, $\gamma > 0$, $d+ \alpha - 2 > 0$, we derive from \eqref{betaValue} that the fact $\beta_1, \beta_2 \ge 0$ is equivalent to the one that  
\begin{equation}\label{pro1-thm1-rad-proof-p1}
\frac{2q d}{d+ \alpha} \le \gamma \le \frac{pd}{d-sp} \mbox{ if } p(d + \alpha) - 2 q (d - sp) > 0 \mbox{ and } d - sp > 0, 
\end{equation}
\begin{equation}\label{pro1-thm1-rad-proof-p2}
\frac{2q d}{d+ \alpha} \le \gamma < + \infty \mbox{ if } p(d + \alpha) - 2 q (d - sp) > 0 \mbox{ and } d - sp \le 0, 
\end{equation}
and
\begin{equation}\label{pro1-thm1-rad-proof-p3}
\frac{pd}{d-sp} \le \gamma \le \frac{2q d}{d+ \alpha}  \mbox{ if } p(d + \alpha) - 2 q (d - sp) < 0 \mbox{ and } d - sp > 0.   
\end{equation}
(Note that the case $p(d + \alpha) - 2 q (d - sp) < 0$ and  $d - sp \le 0$ does not occur.)

We next claim that 
\begin{equation}\label{pro1-thm1-rad-proof-p5}
\gamma_{\rm{rad}} - \frac{2 q d}{d + \alpha} = \frac{(d-\alpha)}{(d+\alpha)(sp(d+\alpha-2)+(d-\alpha))} \big(p(d+\alpha)-2q(d-sp) \big). 
\end{equation}
Indeed, we have, by \eqref{gamaRad}, 
\begin{align*}
\gamma_{\rm{rad}} -\frac{2qd}{d+\alpha} &= q+ \frac{\big(q(sp-1)+p \big)(d-\alpha)}{sp(d+\alpha-2)+(d-\alpha)}- \frac{2 q d}{d + \alpha} \\ 
&=  \frac{\big(q(sp-1)+p \big)(d-\alpha)}{sp(d+\alpha-2)+(d-\alpha)}- \frac{q(d-\alpha)}{d + \alpha} \\ 
&=  \frac{(d-\alpha) \Big( qsp(d+\alpha) + (p-q) (d+\alpha)  -q s p (d+\alpha-2)-q(d-\alpha)\Big)}{(d+\alpha)\big(sp(d+\alpha-2)+(d-\alpha)\big)},  
\end{align*}
which yields claim \eqref{pro1-thm1-rad-proof-p5}.

In the case 
$$
\beta_1\gamma +   \frac{\beta_2 \gamma (d+\alpha-2)}{d-1} > 1,
$$
using from \eqref{pro1-thm1-rad-proof-p0}-\eqref{pro1-thm1-rad-proof-p5}, we have the following three situations: 

\begin{itemize}

\item if $\cD > 0$ and $s p < d$, then 
$$
\gamma_{\rm{rad}} < \gamma \le \frac{pd}{d-sp}
$$

\item if $\cD > 0$ and $s p \ge  d$ (this is equivalent to the fact $sp \ge d$), then 
$$
\gamma_{\rm{rad}} < \gamma < + \infty.
$$

\item if $\cD <  0$ and $s p <  d$ (this is equivalent to the fact $\cD<0$), then 
$$
\frac{pd}{d-sp} \le \gamma < \gamma_{\rm{rad}}.
$$
\end{itemize}

In the case 
$$
\beta_1\gamma +   \frac{\beta_2 \gamma (d+\alpha-2)}{d-1} = 1 \quad \mbox{ and } \quad q = \frac{p}{1-sp} 
$$
we have 
$$
\gamma = \gamma_{\rm{rad}}. 
$$
The conclusion follows.  The proof is complete. 
\end{proof}

Depending on the range where $\gamma$ belongs, we intend to employ different techniques to prove \Cref{thm1-rad}. Using \Cref{pro1-thm1-rad-proof}, we divide the range of $\gamma$ into three main sub-ranges. 

\begin{align}\tag{Range A}\label{gammaRange1}
\begin{cases}
\dsp \gamma_{\mathrm{rad}} < \gamma \leq \frac{pd}{d-sp}, &\text{ if } \cD > 0 \mbox{ and } sp<d, \\[6pt] 
\dsp \gamma_{\mathrm{rad}}< \gamma <\infty, &\text{ if } sp \geq d. 
\end{cases}
\end{align}

\begin{align}\tag{Range B}\label{gammaRange2}
\dsp \frac{pd}{d-sp} \leq \gamma < \gamma_{\mathrm{rad}}, \quad &\text{ if } \cD < 0.
\end{align}
 
\begin{align}\tag{Range C}\label{gammaRange3}
\dsp \gamma = \gamma_{\rm{rad}} \text{ and } q= \frac{p}{1-sp}.
\end{align}

We next recall the following result from \cite{NguMalJFA}.

 \begin{lem} $($\cite[Lemma 2.5]{NguMalJFA}$)$ \label{pro1-thm} Let $d \ge 2$, $0 \le s \le 1$, $1 \leq \gamma, \, p, \, q <\infty$, $0<\alpha<d$, and $0 \le \beta_1, \, \beta_2 < + \infty$. Set 
\begin{equation}\label{gammaCS}
\gamma_{cs}:= \frac{p(\alpha+2qs)}{\alpha+sp}.
\end{equation}
and assume that \eqref{betaEquations} holds and $(d+\alpha)p-2q(d-sp)\neq 0$.   Then \eqref{ConCondtn} is equivalent to the fact 
\begin{align}\label{gammaRange-proof}
\begin{cases}
\dsp \gamma_{cs} \leq \gamma \leq \frac{pd}{d-sp} &\text{ if } \cD >0 \mbox{ and } sp < d, \\[6pt]
\dsp  \gamma_{cs} \leq \gamma <\infty  &\text{ if } sp\geq d, \\[6pt]
\dsp \frac{pd}{d-sp} \leq \gamma \leq  \gamma_{cs} &\text{ if } \cD < 0.
\end{cases}
\end{align}
\end{lem}

\begin{rem}\label{intPol}
Let $d \ge 2$, $0< s \le 1$, $1<\gamma<+\infty, \ 1\le p, \, q < + \infty$, $1 <\alpha<d$, $0< \beta_1, \, \beta_2 < + \infty$ be such that \eqref{betaEquations} and $\cD \neq 0$ hold. 
One can check that  
\begin{equation} 
\mbox{$\gamma_{\rm{cs}}$ satisfies \eqref{betaEquations}  and  } \beta_1\gamma_{\rm{cs}} + \beta_2 \gamma_{\rm{cs}} \label{gammaCSeq}= 1
\end{equation}
(see, e.g., the proof of \cite[Lemma 2.5]{NguMalJFA}).  Using \eqref{gammaCSeq}, we obtain 
\begin{equation}
\beta_1\gamma_{cs} + \frac{\beta_2 \gamma_{cs} (d+\alpha-2)}{d-1} - 1 = \frac{\beta_2\gamma_{\rm{cs}}(\alpha-1)}{d-1}. 
\end{equation}
Considering \eqref{pro1-thm1-rad-proof-p0} with $\gamma = \gamma_{cs}$ and using the above identity, we obtain  
\begin{equation}
\frac{\beta_2\gamma_{\rm{cs}}(\alpha-1)}{d-1}=  \frac{sp(d+\alpha-2)+(d-\alpha)}{(d-1)\bct{p(d+\alpha)-2q(d-sp)}} \bct{\gamma_{\rm{cs}}- \gamma_{rad}}. \label{gammaCSandRad}
\end{equation}
 Since $1<\alpha<d$ and $\beta_2> 0$, it follows from \eqref{gammaCSandRad} that 
\begin{equation} \label{intPol-p1}
\mbox{$\gamma_{\rm{rad}}<\gamma_{\rm{cs}}$ if  $\cD>0$}
\end{equation}
and 
\begin{equation} \label{intPol-p2}
\mbox{$\gamma_{\rm{cs}}<\gamma_{\rm{rad}}$ if $\cD <0$.}
\end{equation}
Applying \Cref{thm1} and using \Cref{pro1-thm}, we deduce that \eqref{CSF1-*} holds 
if  $\gamma$ satisfies \eqref{gammaRange-proof}. Consequently,  \eqref{CSF1} holds if $\gamma$ satisfies \eqref{gammaRange-proof}. Using this fact, we only need to prove \Cref{thm1-rad} for $\gamma_{\rm{rad}}<\gamma<\gamma_{\rm{cs}}$, when $\gamma$ varies in \ref{gammaRange1} and for $\gamma_{\rm{cs}}<\gamma<\gamma_{\rm{rad}}$ when $\gamma$ varies in \ref{gammaRange2}.
\end{rem}

\subsubsection{Proof of  \Cref{thm1-rad} when $\gamma$ varies in the \ref{gammaRange1}} 
As mentioned in \Cref{intPol}, one only needs to establish \eqref{CSF1} for 
$$
\gamma_{\rm{rad}} < \gamma < \gamma_{cs}, 
$$ 
which is assumed from later on in this part. 

 Let $\epss>0$ be small enough so that $p(d+\alpha-2\epss) -2q(d-sp)>0.$ This can be done since $\cD>0$.
 Set 
 $$
 \gamma_\eps := q+\frac{\big(q(sp-1)+p \big)(d-\alpha+2\eps)}{sp(d+\alpha-2\epsilon-2) +(d-\alpha+2\epss)}.
$$   
 We claim that \eqref{CSF1} holds for $\gamma = \gamma_\eps$ for $\eps$ sufficiently small. 

We first admit this claim and continue the proof.  We have, by \eqref{gamaRad}, 
 \begin{align}\label{gama-gamaRad}
\gamma_\eps -\gamma_\mathrm{rad}&=  \frac{(q(sp-1)+p)(d-\alpha+2\eps)}{sp(d+\alpha-2\epsilon-2) +(d-\alpha+2\epss)}- \frac{\bct{q(sp-1) +p}(d-\alpha)}{sp(d+\alpha-2)+(d-\alpha)}\notag \\ &= \frac{4\eps sp(d-1)\big(p+q(sp-1) \big)}{\big(sp(d+\alpha-2\epsilon-2) +(d-\alpha+2\epss)\big) \big(sp(d+\alpha-2) +(d-\alpha)\big)}.
\end{align}

We first claim that 
\begin{equation}\label{thm1-A-p1}
p+(sp-1)q>0. 
\end{equation}
Indeed, \eqref{thm1-A-p1} is clear for $sp\geq 1$. We now establish  \eqref{thm1-A-p1} for $sp<1$.  Since  $\cD = p(d + \alpha ) - 2 q (d - sp ) >0$, we obtain using $\alpha<d$
\begin{align}\label{temp1}
q<\frac{p(d+\alpha)}{2(d-sp)}< \frac{dp}{d-sp}.
\end{align}
We have
\begin{equation*}
\frac{p}{1-sp}- \frac{dp}{d-sp}= \frac{p\bct{d-sp-d+dsp}}{(1-sp)(d-sp)} = \frac{sp^2 (d-1)}{(1-sp)(d-sp)}. 
\end{equation*}
Claim \eqref{thm1-A-p1} then follows in the case $sp < 1$. 

Thus \eqref{gama-gamaRad} implies $\gamma_\eps \to \gamma_{\rm{rad}}+$ as $\eps\to0+$. Since \eqref{CSF1} holds for $\gamma = \gamma_{cs}$ by \eqref{gammaCSeq} and \Cref{thm1}, using interpolation (H\"older's inequalities), we derive from the claim that \eqref{CSF1} holds for $\gamma_{\rm{rad}} < \gamma < \gamma_{cs}$.

It hence remains to show that \eqref{CSF1} holds for $\gamma = \gamma_\eps$ for $\eps$ sufficiently small. Let $g\in C_{c, \rm{rad}}^1(\ro^d)$. Since $p(d+\alpha)-2q(d-sp)\neq0$ and \eqref{CSF1} is invariant under scaling, without loss of generality, we can assume that 
\begin{align}\label{qBddAbv1A}
\int_{\ro^d} \int_{\ro^d} \frac{|g(x)-g(y)|^p}{|x-y|^{d+sp}}dx dy = 1= \int_{\ro^d} \int_{\ro^d}\frac{|g(x)|^q|g(y)|^q}{|x-y|^{d-\alpha}}dxdy.
\end{align}
By \eqref{plsEpss} in \Cref{Rad}, it suffices to prove that 
\begin{align}\label{qBddAbv1}
\int_{\ro^d} \frac{|g(x)|^q dx}{|x|^{\frac{d-\alpha}{2}+\epss}} \leq C,
\end{align}
for some constant $C>0$ independent of $g$. Because of \eqref{ruizPlsEpss} with $R = 1 $ in the \Cref{ruiz} below, we only need to show that 
\begin{equation} \label{thm1-A-p00}
\int_{|x|<1} \frac{|g(x)|^q dx}{|x|^{\frac{d-\alpha}{2}+\epss}} < C, 
\end{equation}
for some constant $C>0$ independent of $g$. 

Since $p+(sp-1)q>0$ by \eqref{thm1-A-p1}, we derive from \eqref{gamaRad} that 
$$
q<\gamma_{\mathrm{rad}}.
$$  
This implies, by \eqref{intPol-p1}, that 
$$
q < \gamma_{cs}. 
$$ 
Since  \eqref{CSF1} holds for $\gamma = \gamma_{cs}$ by \eqref{gammaCSeq} and \Cref{thm1}, it follows from  \eqref{qBddAbv1A} that 
\begin{equation} \label{thm1-A-p2}
\int_{|x| < 1} |g(x)|^{\gamma_{cs}} \, dx \le C. 
\end{equation}
Applying H\"older's inequality, we derive that 
\begin{multline}\label{thm1-A-p3}
\int_{|x|<1} \frac{|g(x)|^q dx}{|x|^{\frac{d-\alpha}{2}+\epss}} \leq \left( \int_{|x| < 1} |g(x)|^{\gamma_{cs}} \, d x \right)^{\frac{q}{\gamma_{cs}}} \bct{\int_{|x|<1}  \frac{dx}{\md{x}^{\frac{\gamma_{cs}}{\gamma_{cs}-q} \bct{\frac{d-\alpha}{2} +\epss}  }}}^\frac{\gamma_{cs}-q}{\gamma_{cs}} \\[6pt]
\mathop{\le}^{\eqref{thm1-A-p2}} C \bct{\int_{|x|<1}  \frac{dx}{\md{x}^{\frac{\gamma_{cs}}{\gamma_{cs}-q} \bct{\frac{d-\alpha}{2} +\epss}  }}}^\frac{\gamma_{cs}-q}{\gamma_{cs}}. 
\end{multline}
Since 
$$
d- \frac{\gamma_{cs}}{\gamma_{cs}-q} \bct{\frac{d-\alpha}{2} +\epss} = \frac{1}{\gamma_{cs}-q} \left[  \frac{\alpha\big( p(d+\alpha) -2q(d-sp) \big) }{2(\alpha+sp)}-\gamma_{cs}\epss \right]>0,
$$ for $\epss>0$ small enough, assertion \eqref{thm1-A-p00} follows for $\eps $ sufficiently small. The proof is complete. \qed

\medskip 

The following result  due to Ruiz \cite{Rui} is used in the proof.  
 
\begin{lemma}(\cite[Theorem 1.1]{Rui})\label{ruiz}
Let $d\in \mathbb N$, $0<\alpha<d$, $1\leq q<\infty$. Then for every $\epss>0$ and $R>0$ there exists $C= C(d,\alpha,q,\epss)>0$ such that for all $u\in L^\frac{2dq}{d+\alpha}(\ro^d)$ we have, 

\begin{align}
\int_{|x|>R} \frac{|u(x)|^q}{|x|^{\frac{d-\alpha}{2}+\epss}} dx &\leq \frac{C}{R^\epss} \bct{\int_{\ro^d} \int_{\ro^d} \frac{|u(x)|^q|u(y)|^q}{|x-y|^{d-\alpha}} dx dy}^\frac{1}{2} \label{ruizPlsEpss} \mbox{ and } \\ 
\int_{|x|<R} \frac{|u(x)|^q}{|x|^{\frac{d-\alpha}{2}-\epss}} dx &\leq CR^\epss \bct{\int_{\ro^d} \int_{\ro^d} \frac{|u(x)|^q|u(y)|^q}{|x-y|^{d-\alpha}} dx dy}^\frac{1}{2} \label{ruizMnsEpss}.
\end{align}
\end{lemma}

\subsubsection{Proof of  \Cref{thm1-rad} when $\gamma$ varies in the \ref{gammaRange2}}

As mentioned in \Cref{intPol}, it suffices to prove \eqref{CSF1} for 
$$
\gamma_{\rm{cs}}<\gamma<\gamma_{\rm{rad}}, 
$$
which will be assumed from later on in this part. 

Let $g\in C^1_{c,\rm{rad}}(\ro^d)$.
Since $p(d+\alpha)-2q(d-sp)\neq0$ and \eqref{CSF1} is invariant under scaling, without loss of generality we can assume that 
\begin{align}\label{qBddAbv1B}
\int_{\ro^d} \int_{\ro^d} \frac{|g(x)-g(y)|^p}{|x-y|^{d+sp}}dx dy = 1= \int_{\ro^d} \int_{\ro^d}\frac{|g(x)|^q|g(y)|^q}{|x-y|^{d-\alpha}}dxdy.
\end{align}

\medskip 
We now consider three cases separately: 

\begin{itemize}
\item Case 1:  $1 < sp < d$. 
\item Case 2:  $sp \leq 1$ and $q(sp-1)+p > 0$.  
\item Case 3: $sp \leq 1$ and $q(sp-1)+p \le 0$. 
\end{itemize}

\medskip 
We now proceed the proof. 

\medskip 

$\bullet$ Case 1: $1 < sp < d$.  Set, for $\eps$ sufficiently small, 
$$
\gamma_\eps= q+\frac{(q(sp-1)+p)(d-\alpha-2\eps)}{sp(d+\alpha+2\epsilon-2) +(d-\alpha-2\epss)}.
$$
 We claim that \eqref{CSF1} holds for $\gamma = \gamma_\eps$ for $\eps$ sufficiently small. 

We first admit this claim and continue the proof.   We have, by \eqref{gamaRad}, 
\begin{align}
\gamma_\eps-\gamma_{\rm{rad}} &= \frac{(q(sp-1)+p)(d-\alpha-2\eps)}{sp(d+\alpha+2\epsilon-2) +(d-\alpha-2\epss)}-\frac{(q(sp-1)+p)(d-\alpha)}{sp(d+\alpha-2) +(d-\alpha)} \notag \\ &= -\frac{4\eps sp(d-1)\big(p+q(sp-1) \big)}{\bct{sp(d+\alpha+2\epsilon-2) +(d-\alpha-2\epss)} \bct{sp(d+\alpha-2) +(d-\alpha)}} \label{gamaEps-gamaRad2}. 
\end{align}
It follows that  $\gamma_\eps\to \gamma_{\rm{rad}}-$ and $\eps\to 0+$ since $sp  > 1$. Now, since 
\eqref{CSF1} holds with $\gamma= \gamma_{cs}$ by \eqref{gammaCSeq}, by interpolation, it suffices to prove the claim. 

Next, we establish the claim. Applying \eqref{mnsEpss} in \Cref{Rad}, it suffices to show that 
\begin{align}\label{q,spBddBlw1}
\int_{\ro^d} \frac{|g(x)|^q}{\md{x}^{\frac{d-\alpha}{2}-\epss}} dx \leq C,
\end{align}
for some constant $C>0$ independent of $g$ and for $\epss>0$ small enough. Using \eqref{ruizMnsEpss} in \Cref{ruiz}, one only needs to prove 
\begin{align}\label{q,spBddBlw1-00}
\int_{|x| > 1} \frac{|g(x)|^q}{\md{x}^{\frac{d-\alpha}{2}-\epss}} dx \leq C. 
\end{align}

We have, by \eqref{gamaRad}, 
$$
q<\gamma_{\rm{rad}}.
$$ 

If  $\gamma_{cs}<q<\gamma_{\mathrm{rad}}$, using \eqref{StLemIneq} and \eqref{qBddAbv1B}, we obtain 
$$
\int_{|x|>1} \frac{|g(x)|^q}{\md{x}^{\frac{d-\alpha}{2}-\epss}} dx \le C \int_{|x| > 1} \frac{d x}{|x|^{\left(\frac{d-\alpha}{2}-\epss \right) + \frac{q(d-sp)}{p} }}
\leq C,  
$$
for $\eps$ sufficiently small since, by $\cD< 0$,  
$$
\frac{d-\alpha}{2} + \frac{q(d-sp)}{p}  > d. 
$$
 
If $q=\gamma_{\rm{cs}}$, then applying \eqref{CSF1} with $\gamma= \gamma_{cs}$ by \eqref{gammaCSeq},  and using \eqref{qBddAbv1B}, we have
$$
\int_{|x|>1} \frac{|g(x)|^q}{|x|^{\frac{d-\alpha}{2}-\eps}} dx \leq \int_{|x|>1}|g(x)|^{\gamma_{\rm{cs}}} dx \leq C.
$$

If $q<\gamma_{cs} <\gamma_{\mathrm{rad}}$, applying \eqref{CSF1} with $\gamma= \gamma_{cs}$ by \eqref{gammaCSeq},  and using \eqref{qBddAbv1B}, we have
\begin{multline}\label{Holder-minus-eps}
\int_{|x|>1} \frac{|g(x)|^q }{|x|^{\frac{d-\alpha}{2}-\epss}} d x \leq \left(\int_{|x|>1} |g(x)|^{\gamma_{cs}} \, dx \right)^{\frac{q}{\gamma_{cs}}} \bct{\int_{|x|>1}  \frac{dx}{\md{x}^{\frac{\gamma_{cs}}{\gamma_{cs}-q} \bct{\frac{d-\alpha}{2} -\epss}  }}}^\frac{\gamma_{cs}-q}{\gamma_{cs}} \\[6pt]
\le C \bct{\int_{|x|>1}  \frac{dx}{\md{x}^{\frac{\gamma_{cs}}{\gamma_{cs}-q} \bct{\frac{d-\alpha}{2} -\epss}  }}}^\frac{\gamma_{cs}-q}{\gamma_{cs}}. 
\end{multline}
 Since $\cD = p(d+\alpha)-2q(d-sp)<0$, it follows that 
 $$d- \frac{\gamma_{cs}}{\gamma_{cs}-q} \bct{\frac{d-\alpha}{2} -\epss} = \frac{1}{\gamma_{cs}-q} \left[  \frac{\alpha\bct{p(d+\alpha) -2q(d-sp)} }{2(\alpha+sp)}+\gamma_{cs}\epss \right]<0,$$ for $\epss>0$ small enough. It follows from \eqref{Holder-minus-eps} that 
\begin{equation}\label{Holder-minus-eps-00}
 \int_{|x|>1} \frac{|g(x)|^q dx}{|x|^{\frac{d-\alpha}{2}-\epss}} \le C. 
\end{equation}
The proof is complete in this case.

\medskip 
$\bullet$ Case 2: $sp \le 1$ and $p + q(sp - 1) > 0$. Set 
$$
\gamma=\gamma_\eps= q+\frac{\big(q(sp-1)+p \big)(d-\alpha-2\eps)}{sp(d+\alpha+2\epsilon-2) +(d-\alpha-2\epss)}, 
$$ 
for $\eps$ sufficiently small. 

From \eqref{gamaRad},  we derive that $\gamma_\eps \to \gamma_{\rm{rad}}-$ as $\eps\to 0+$ (see for e.g \eqref{gamaEps-gamaRad2}). Since 
\eqref{CSF1} holds with $\gamma= \gamma_{cs}$ by \eqref{gammaCSeq}, by interpolation, it suffices to prove \Cref{thm1-rad} for $\gamma_\eps$ with $\eps$ sufficiently small. Applying \eqref{mnsEpss}, we only need to establish
\begin{align}\label{q,spBddBlwAbv1}
\int_{\ro^d} \frac{|g(x)|^q}{\md{x}^{\frac{d-\alpha}{2}-\epss}} dx \leq C,
\end{align}
for some constant $C>0$ independent of $g$ and for $\epss>0$ small enough. Applying \eqref{ruizMnsEpss} in \Cref{ruiz}, we only need to show
\begin{align}\label{q,spBddBlwAbv1-00}
\int_{|x|>1} \frac{|g(x)|^q}{\md{x}^{\frac{d-\alpha}{2}-\epss}} dx \leq C. 
\end{align}
If $q<\frac{dp}{d-sp}$, then $q<\gamma_{\rm{cs}}$ given by \eqref{gammaCS} since $\cD < 0$ and $sp < d$. Similar to 
\eqref{Holder-minus-eps-00}, we obtain \eqref{q,spBddBlwAbv1-00}.
If  $\frac{dp}{d-sp}\leq q$ then it follows, since $p + q (sp - 1) > 0$, that 
$$
\frac{1}{p}-s < \frac{1}{q} \le \frac{1}{p}-\frac{s}{d}.
$$ 
Applying \Cref{wkNi} with $\gamma=q$ and $R = 1$, we have
\begin{equation}\label{q,spBddBlwAbv1-0000}
\int_{|x| > 1} |g(x)|^q \, dx \le C. 
\end{equation}
Assertion \eqref{q,spBddBlwAbv1-00} now follows from  \eqref{q,spBddBlwAbv1-0000} by noting that $\frac{1}{\md{x}^{\frac{d-\alpha}{2}-\epss}} \le 1$ for $|x| > 1$. 

The proof of \eqref{q,spBddBlwAbv1-00} is complete.

$\bullet$ Case 3: $sp \leq 1$ and $p  + q (sp - 1) \le 0$. Since, $p\geq 1$, so we must have $sp<1$ in this case.

We first consider the case $p + q (sp -1) = 0$. Then $\gamma_{\rm{rad}} = q$ by \eqref{gamaRad}. Applying \Cref{wkNi} and using \eqref{qBddAbv1B}, we have
\begin{align}\label{thm1-C-p11}
\int_{|x|>1}|g|^q \leq C. 
\end{align}
On the other hand, by \eqref{qBddAbv1B}, 
\begin{equation} \label{thm1-C-p22}
\int_{|x|<1} |g|^q dx \leq C. 
\end{equation}
From \eqref{thm1-C-p11} and \eqref{thm1-C-p22}, we obtain 
$$
\int_{\R^d} |g(x)|^q \, dx \le C. 
$$ 
Since \eqref{CSF1} holds with $\gamma= \gamma_{cs}$ by \eqref{gammaCSeq}, by interpolation, the conclusion holds for $\gamma_{cs} < \gamma < \gamma_{\rm{rad}} = q$.

\medskip 

We next deal with the case $p + q (sp - 1) < 0$. 
Set
\begin{equation}\label{thm1-B2-gamma3}
\gamma_\eps= q+\frac{\big(q(sp-1)+p\big)(d-\alpha+2\eps)}{sp(d+\alpha-2\epsilon-2) +(d-\alpha+2\epss)}, 
\end{equation}
for $\eps$ sufficiently small. Since  $p + q (sp-1) <0$, it follows from \eqref{gama-gamaRad} that 
$\gamma_\eps \to \gamma_{\rm{rad}}-$ as $\eps\to 0+$. Since 
\eqref{CSF1} holds with $\gamma= \gamma_{cs}$ by \eqref{gammaCSeq}, by interpolation, it suffices to prove \Cref{thm1-rad} for $\gamma_\eps$, i.e.,  to prove 
$$
\int_{\mR^d} |g(x)|^{\gamma_{\eps}} \, dx < C, 
$$
for  $\eps$ sufficiently small.

We have, by \eqref{qBddAbv1B}, 
$$
\int_{|x| < 1} |g(x)|^q \, dx \le C. 
$$
Since $\gamma_{\eps} < \gamma_{\rm{rad}} \le  q$ thanks to the fact $p + q (s p - 1) \le 0$, it follows from H\"older's inequality that 
$$
\int_{|x| < 1} |g(x)|^{\gamma_\eps} \, dx \le C. 
$$
It remains to prove that 
$$
\int_{|x|  >  1} |g(x)|^{\gamma_\eps} \, dx \le C. 
$$

Set 
\begin{equation}\label{thm1-B2-r}
r=\frac{p}{1-sp}.
\end{equation}
We have, by \eqref{gamaRad},  
\begin{align*}
\gamma_{\rm{rad}}-r&= q + \frac{\big(q(sp-1)+p \big)(d-\alpha)}{sp(d+\alpha-2)+(d-\alpha)}  -\frac{p}{1-sp}\\&= \big(q(sp-1)+p \big) \bct{\frac{d-\alpha}{sp(d+\alpha-2)+(d-\alpha)} - \frac{1}{1-sp}}\\&
= -\frac{2(q(sp-1)+p)sp(d-1)}{\big(sp(d+\alpha-2)+(d-\alpha)\big)(1-sp)}. 
\end{align*}
Since $p + q(sp -1) <  0$, it follows that 
\begin{equation} \label{thm1-B2-p1}
r < \gamma_{\rm{rad}} \mathop{<}^\eqref{gamaRad} q. 
\end{equation}
For $\eps > 0$ sufficiently small, set 
\begin{equation}\label{thm1-B2-eta}
\eta = \eta_\eps : = \frac{d-\alpha}{2}+\epss 
\end{equation}
and let $\theta = \theta_\eps \in (0, 1)$ be such that 
$$
\frac{\theta}{r}+\frac{1-\theta}{q} =\frac{1}{\gamma_\eps}
$$
(the existence of $\theta$ follows from \eqref{thm1-B2-p1}). One has 
\begin{equation}\label{thm1-B2-theta}
\theta= \frac{q-\gamma_\eps}{q-r} \frac{r}{\gamma_\eps} \quad \mbox{ and } \quad 1-\theta = \frac{(\gamma_\eps-r)q}{(q-r)\gamma_\eps}.
\end{equation}

We have, by H\"older's inequality,  
\begin{multline}
\int_{|x|>1} |g(x)|^{\gamma_\eps} \, dx = \int_{|x|>1} \bct{|g(x)|^{\theta \gamma_\eps}|x|^\frac{\eta \gamma_\eps(1-\theta)}{q}} \bct{|g(x)|^{(1-\theta )\gamma_\eps}\frac{1}{|x|^{\frac{\eta \gamma_\eps(1-\theta)}{q}}}} \, dx \\[6pt] 
\leq \bct{\int_{|x|>1} |g(x)|^r |x|^\frac{r \eta (1 - \theta) }{\theta q} \, dx}^\frac{\gamma_\eps \theta}{r} \bct{\int_{|x|>1} \frac{|g|^q}{|x|^\eta}\, dx }^\frac{(1-\theta)\gamma_\eps}{q}, 
 \end{multline}
 which yields, by \eqref{thm1-B2-eta}, \eqref{ruizPlsEpss}, and \eqref{qBddAbv1B}  
\begin{equation}
\int_{|x|>1} |g(x)|^{\gamma_\eps} \, dx  \le C \bct{\int_{|x|>1} |g(x)|^r |x|^\frac{r \eta (1 - \theta) }{\theta q} \, dx}^\frac{\gamma_\eps \theta}{r} = C \bct{\int_{|x|>1} |g|^r |x|^{-r\beta}dx}^\frac{\gamma_\eps \theta_\eps}{r}, \label{qgeqGmarad}
\end{equation}
by \eqref{thm1-B2-theta},  where 
\begin{equation} \label{thm1-B2-beta1}
\beta = \beta_\eps : = -\frac{(\gamma_\eps-r)\eta}{(q-\gamma_\eps)r}
\end{equation}

We claim that 
\begin{equation} \label{thm1-B2-beta2}
\beta =-(d-1)s.
\end{equation}

Since, by \eqref{thm1-B2-r},  
$$
q(sp-1)+p = (q-r)(sp-1), 
$$
it follows from \eqref{thm1-B2-gamma3} that 
\begin{equation}\label{thm1-B2-gamma3-2}
\gamma_\eps= q+ \frac{(q-r)(d-\alpha+2\eps)(sp-1)}{sp(d+\alpha-2\epsilon-2) +(d-\alpha+2\epss)}. 
\end{equation}
This implies 
\begin{equation*}
\gamma_\eps-r= \bct{q-r}\bct{1+ \frac{(d-\alpha+2\eps)(sp-1)}{sp(d+\alpha-2\epsilon-2) +(d-\alpha+2\epss)}}, 
\end{equation*}
which yields 
\begin{equation}\label{gamma-r}
\gamma_\eps-r= \frac{2(q-r)sp(d-1)}{sp(d+\alpha-2\epsilon-2) +(d-\alpha+2\epss)}.
\end{equation}
We also have, by \eqref{thm1-B2-gamma3-2}
\begin{equation}\label{q-gamma}
q-\gamma_\eps =-\frac{(q-r)(sp-1)(d-\alpha+2\eps)}{sp(d+\alpha-2\epsilon-2) +(d-\alpha+2\epss)}.
\end{equation}
Combining \eqref{thm1-B2-r}, \eqref{gamma-r}, and \eqref{q-gamma} yields 
\begin{equation*}
\beta_\eps = -\frac{(\gamma_\eps-r)\eta_\eps}{(q-\gamma_\eps)r} = -(d-1)s.
\end{equation*}

Applying \Cref{stWss} with $\beta=-(d-1)s$ and using \eqref{qBddAbv1B} and  \eqref{qgeqGmarad}, we obtain 
\begin{align*} 
\int_{|x|>1} |g|^{\gamma_\eps} dx \leq C. 
\end{align*}

The proof is complete. \qed

\subsubsection{Proof of \Cref{thm1-rad} when $\gamma$ varies in the \ref{gammaRange3}} 

Let $g\in C^1_{c,\rm{rad}}(\ro^d)$.
Since $p(d+\alpha)-2q(d-sp)\neq0$ and \eqref{CSF1} is invariant under scaling, without loss of generality we can assume that 
\begin{align}\label{qBddAbv1B-111}
\int_{\ro^d} \int_{\ro^d} \frac{|g(x)-g(y)|^p}{|x-y|^{d+sp}}dx dy = 1= \int_{\ro^d} \int_{\ro^d}\frac{|g(x)|^q|g(y)|^q}{|x-y|^{d-\alpha}}dxdy.
\end{align}
We first note that $\gamma = q$ in this case.  
Applying \Cref{wkNi} and using \eqref{qBddAbv1B-111}, we have
\begin{align}\label{thm1-C-p1}
\int_{|x|>1}|g|^q \leq C. 
\end{align}
On the other hand, by \eqref{qBddAbv1B-111}, 
\begin{equation} \label{thm1-C-p2}
\int_{|x|<1} |g|^q dx \leq C. 
\end{equation}
The conclusion follows from \eqref{thm1-C-p1} and \eqref{thm1-C-p2}. 
\qed

\section{Optimality of the Range of Parameters} \label{sect-Opt}
In this section, we prove \Cref{thm1-rad-opt} and \Cref{thm1-rad-opt1}. Let $\eta \in C_c^\infty(\ro \setminus \{0\})$ be a nontrivial, nonnegative function with $\rm{spt}\ \eta \subset (-1,1)$. Define
\begin{equation}\label{fam-func}
g_{\lambda, R, S}(x): = \lambda \eta\bct{\frac{|x|-R}{S}}, \forall R>S>0 \text{ and } \lambda>0.  
\end{equation}
Applying \cite[Theorem 1]{BM-GNI}, we have
\begin{equation}\label{GNI-*}
\nrm{g}_{\dot W^{s,p}(\ro^d)} \lesssim \nrm{\nabla g}_{L^p(\ro^d)} +  \nrm{g}_{L^p(\ro^d)} \ \forall g\in C_c^\infty(\ro^d).
\end{equation}
Here and in what follows in this section,  $a \lesssim b$ means that $a \le C b$ for some positive constant $C$ depending only on the parameters in \Cref{thm1-rad-opt} and \Cref{thm1-rad-opt1}.

By a scaling argument, we derive from  \eqref{GNI-*} that 
\begin{equation}\label{GNI}
\nrm{g}_{\dot W^{s,p}(\ro^d)} \lesssim \nrm{g}_{\dot W^{1,p}(\ro^d)}^s \nrm{g}_{L^p(\ro^d)}^{1-s}, \ \forall g\in C_c^\infty(\ro^d).
\end{equation}
Using \eqref{GNI} we can estimate 
\begin{equation}\label{WspEst}
\nrm{g_{\lambda,R,S}}^p_{\dot W^{s,p}(\ro^d)} \lesssim \lambda^p  R^{d-1} S^{1-sp}, \text{ for } 0<s<1 \text{ and } 1\leq p<+\infty.
\end{equation}

On the other hand,  one has, see e.g.,  \cite[(5.8)]{MorSch'18}, 
\begin{equation}\label{CouEst}
  \int_{\ro^d}\int_{\ro^d} \frac{|g_{\lambda,R, S}(x)|^q|g_{\lambda, R, S}(y)|^q}{|x-y|^{d-\alpha}}dxdy \lesssim \begin{cases}\lambda^{2q}R^{d+\alpha-2 }S^2 &\text{ if } 1<\alpha<d, \\ \lambda^{2q} R^{d-1}S^2\log(R/S) &\text{ if } \alpha =1, \\ \lambda^{2q} R^{d-1}S^{1+\alpha} &\text{ if } 0<\alpha<1.
  \end{cases}
\end{equation}
A straightforward computation gives 
\begin{equation}\label{LGammaEst}
\lambda^\gamma \bct{R-S}^{d-1}S \lesssim \nrm{g_{\lambda, R, S}}^\gamma_{L^\gamma(\ro^d)}. 
\end{equation}

We are now ready to give the proofs of \Cref{thm1-rad-opt} and \Cref{thm1-rad-opt1}. We will only focus on the case $0<s<1$. The cases $s=1$ and $s=0$ are either trivial or follow by the same arguments. 

\subsection{Proof of \Cref{thm1-rad-opt}} We establish the conclusion by considering two cases separately: 
\begin{itemize}
\item  {\it Case 1: $\beta_1 \gamma + \frac{d+\alpha-2}{d-1} \beta_2 \gamma < 1$.}

\item  {\it Case 2: $\beta_1 \gamma + \frac{d+\alpha-2}{d-1} \beta_2 \gamma = 1$ and $q(sp-1)+p \neq 0$.}  
\end{itemize}

\medskip 
\noindent {\it Case 1: $\beta_1 \gamma + \frac{d+\alpha-2}{d-1} \beta_2 \gamma < 1$.} In this case, if \eqref{CSF1} holds,  then we can apply it to $g=g_{1,R, 1}$ for large $R$. Therefore, using estimates \eqref{WspEst}, \eqref{CouEst}, \eqref{LGammaEst},  and the inequality \eqref{CSF1} we have, for large $R$,  
\begin{equation}\label{subCritRange}
R^{d-1} \lesssim R^{(d-1)\beta_1\gamma + (d+\alpha-2)\beta_2\gamma}
\end{equation}
Taking $R\to \infty $ in \eqref{subCritRange},  we obtain a contradiction. 

\medskip 
\noindent {\it Case 2: $\beta_1 \gamma + \frac{d+\alpha-2}{d-1} \beta_2 \gamma = 1$ and $q(sp-1)+p \neq 0$.}

Since, in addition, $\gamma$ satisfy \eqref{betaEquations} so we can compute (see for e.g. \eqref{pro1-thm1-rad-proof-p0})
\begin{equation}\label{gammaRadTemp}
\gamma = q+ \frac{\bct{q(sp-1)+p}(d-\alpha)}{sp(d+\alpha-2)+(d-\alpha)}.
\end{equation}

If \eqref{CSF1} holds, then in particular it holds for 
\begin{equation}\label{testFuncCrit}
g=g_{m,R}= \sum_{k=1}^m g_{R^{k\xi_2},R^k, R^{k\xi_1}},\ \forall R>0\mbox{ and } m\in \mathbb{N},
\end{equation}
where $\xi_1\neq 1$ and $\xi_2\in \ro$ to be chosen subsequently. If $\xi_1>1$, then we start with a small enough $0<R<1$ and if $\xi_1<1$, then we start with a large enough $R>>1$ so that the supports of $g_{R^{k\xi_2},R^k, R^{k\xi_1}}$ remain disjoint for all  $k\in \mathbb{N}$. With this choice and with the help of \eqref{WspEst}, \eqref{CouEst}, \eqref{LGammaEst}, and the inequality \eqref{CSF1},  we derive that 
\begin{multline*}
\sum_{k=1}^m R^{k\gamma \xi_2}\bct{R^k-R^{k\xi_1}}^{d-1} R^{k\xi_1} \\[6pt] \lesssim \bct{\sum_{k=1}^m R^{kp\xi_2}R^{k(d-1)}R^{(1-sp)k\xi_1}}^{\beta_1\gamma}\bct{\sum_{k=1}^m R^{2qk\xi_2}R^{k(d+\alpha-2)}R^{2k\xi_1}}^{\beta_2\gamma}.
\end{multline*}
Simplifying this estimate,  we have 
\begin{multline}\label{main-opt-est}
\sum_{k=1}^m R^{k(\gamma \xi_2+\xi_1+d-1)}\bct{1-R^{k(\xi_1-1)}}^{d-1} \\[6pt]
\lesssim \bct{\sum_{k=1}^m R^{k(p\xi_2+(1-sp)\xi_1+d-1)}}^{\beta_1\gamma}\bct{\sum_{k=1}^m R^{k(2q\xi_2+2\xi_1+d+\alpha-2)}}^{\beta_2\gamma}.
\end{multline}
We now choose $\xi_1$ and $\xi_2$ such that
\begin{equation}\label{xi1xi2}
p\xi_2+(1-sp)\xi_1+(d-1)=0=2q\xi_2+2\xi_1+(d+\alpha-2).
\end{equation}
Since $q(sp-1)+p\neq 0$, the exact expressions of $\xi_1$ and $\xi_2$ can be computed from \eqref{xi1xi2} as follows 
 \begin{equation}\label{xi1xi2EE}
 \xi_1= \frac{2q(d-1)-p(d+\alpha-2)}{2(q(sp-1)+p)} \quad \mbox{ and } \quad  \xi_2 = -\frac{sp(d+\alpha-2)+(d-\alpha)}{2(q(sp-1)+p)}.
 \end{equation}
Notice that using \eqref{gammaRadTemp}, we can simplify the expression of $\xi_2$ as follows 

\begin{equation}\label{xi2AE}
\xi_2= -\frac{d-\alpha}{2(\gamma-q)}.
\end{equation}
Making use of \eqref{xi2AE} and \eqref{xi1xi2} we calculate 
\begin{align}\label{xi1xi2gamma}
\gamma \xi_2+ \xi_1+(d-1)&= (\gamma-q)\xi_2+ q\xi_2+\xi_1+(d-1) \notag \\ &= -\frac{d-\alpha}{2} -\frac{d+\alpha-2}{2}+(d-1)=0
\end{align}
Plugging \eqref{xi1xi2} and \eqref{xi1xi2gamma} into \eqref{main-opt-est} we get 

\begin{equation}\label{final-opt-est}
\sum_{k=1}^m \bct{1-R^{k(\xi_1-1)}}^{d-1} \lesssim m^{\beta_1 \gamma+\beta_2\gamma} ,\ \forall m \in \mathbb{N}.
\end{equation}

Since $\beta_1 \gamma + \frac{d+\alpha-2}{d-1} \beta_2 \gamma = 1$ and $1<\alpha<d$, it follows that $\beta_1 \gamma+ \beta_2\gamma <1$. Taking $R\to + \infty$ in \eqref{final-opt-est} if $\xi_1<1$ and taking $R\to 0+$ in \eqref{final-opt-est} if $\xi_1>1$, we obtain a contradiction.

It remains to show that $\xi_1\neq 1$. Indeed, this  follows from the assumption $p(d+\alpha)-2q(d-sp)\neq 0$ and \eqref{xi1xi2EE}, as $\xi_1$ can be rewritten as $\xi_1= 1-\frac{p(d+\alpha)-2q(d-sp)}{2(q(sp-1)+p)}$. 

The proof is complete. 
\qed

\subsection{Proof of \Cref{thm1-rad-opt1}} We prove the result by contradiction. Assume \eqref{CSF1}. Consider the family of radial functions $g_{1, R, 1}$ for $R>1$ (see \eqref{fam-func} for the definition). Using \eqref{WspEst}, \eqref{CouEst}, and \eqref{LGammaEst}, and applying \eqref{CSF1} to $g_{1,R,1}$, we obtain 
\begin{equation}\label{subCritRange1}
R^{d-1} \lesssim \begin{cases}
R^{(d-1)(\beta_1\gamma +\beta_2\gamma)} , &\mbox{ if } 0<\alpha<1, \\ R^{(d-1)(\beta_1\gamma +\beta_2\gamma)}\bct{(\log(R))^{\beta_2\gamma}+1} , &\mbox{ if } \alpha =1.
\end{cases}
\end{equation}
Taking $R\to + \infty$ in \eqref{subCritRange1} and noticing $\beta_1\gamma+\beta_2\gamma<1$, we obtain a contradiction. The conclusion follows. \qed

\end{document}